\documentclass[12pt,oneside]{amsart}
\usepackage[utf8]{inputenc}
\usepackage{geometry}
\usepackage{mathtools}
\usepackage{amstext}
\usepackage{amsthm}
\usepackage{esint}

\makeatletter
\numberwithin{equation}{section}
\numberwithin{figure}{section}
\theoremstyle{plain}
\newtheorem{thm}{\protect\theoremname}
\theoremstyle{plain}
\newtheorem{lem}[thm]{\protect\lemmaname}

\global\long\def\sign{\operatorname{sign}}

\makeatother

\providecommand{\lemmaname}{Lemma}
\providecommand{\theoremname}{Theorem}

\begin{document}
\title{linear combinations of polynomials with three-term recurrence}
\begin{abstract}
We study the zero distribution of the sum of the first $n$ polynomials
satisfying a three-term recurrence whose coefficients are linear polynomials.
We also extend this sum to a linear combination, whose coefficients
are powers of $az+b$ for $a,b\in\mathbb{R}$, of Chebyshev polynomials.
In particular, we find necessary and sufficient conditions on $a$,
$b$ such that this linear combination is hyperbolic. 
\end{abstract}

\author{Khang Tran}
\address{California State University, Fresno}
\email{khangt@mail.fresnostate.edu}
\author{Maverick Zhang}
\address{University of California, Berkeley}
\email{maverickzhang@berkeley.edu}
\subjclass[2000]{30C15, 26C10}
\maketitle

\section{Introduction}

The sequence of Chebyshev polynomials of the first kind $\left\{ T_{n}(z)\right\} _{n=0}^{\infty}$
defined by the recurrence 
\[
T_{n+1}(z)=2zT_{n}(z)+T_{n-1}(z)
\]
with $T_{0}(z)=1$ and $T_{1}(z)=z$ forms a sequence of orthogonal
polynomials whose zeros are real (i.e., hyperbolic polynomials). The
location of zeros of polynomials satisfying a more general recurrence
\begin{equation}
R_{n+1}(z)=A(z)R_{n}(z)+B(z)R_{n-1}(z)\label{eq:generalrecurrence}
\end{equation}
where $A(z),B(z)\in\mathbb{C}[z]$ was given in \cite{tran}. In \cite{stankov},
the author studied the set of zeros of a linear combination of Chebyshev
polynomials $\sum_{k=0}^{m}a_{k}T_{n-k}(z)$, $m\le n$, $a_{k}\in\mathbb{R},$
and provided a connection between this sequence and the theory of
Pisot and Salem numbers in number theory. In the special case when
$m=n$ and $a_{k}=1$ $\forall k$, the sum of the first $n$ Chebyshev
polynomials connects to Direchlet kernel in Fourier analysis. In Section
2 of this paper, we to study the zeros of this sum (c.f. Theorem \ref{thm:firsttheorem})
when the sequence of Chebyshev polynomials are replaced by a more
general sequence $\left\{ R_{n}(z)\right\} $ given in \eqref{eq:generalrecurrence}
where $A(z)$ and $B(z)$ are any linear polynomials with real coefficients.

The sequence of Chebyshev polynomials of the second kind $\left\{ U_{n}(z)\right\} $
satisfies the same recurrence as that of the first kind with the initial
condition $U_{0}(z)=1$ and $U_{1}(z)=2z$. This initial condition
can be written in the form $U_{0}(z)=1$ and $U_{-n}(z)=0$, $\forall n\in\mathbb{N}$.
In Section 3 of this paper, we study the zeros of a linear combination
of Chebyshev polynomials of the second kind whose coefficients are
power of $az+b$. In particular, we consider 
\begin{equation}
Q_{n}(z)=\sum_{k=0}^{n}(az+b)^{k}U_{n-k}(z),\qquad a,b\in\mathbb{R}.\label{eq:linearcombCheb}
\end{equation}
We find necessary and sufficient conditions on $a$ and $b$ under
which the zeros of resulting polynomials are real (c.f. Theorem \ref{thm:thirdtheorem}).

\section{Sum of polynomials with three-term recurrence}

For $a_{1},b_{1},a_{2},b_{2}\in\mathbb{R}$, $a_{2}\ne0$, we let
$R_{n}(z)$ be the sequence of polynomials satisfying the recurrence
\[
R_{n+1}(z)=(a_{1}z+b_{1})R_{n}(z)+(a_{2}z+b_{2})R_{n-1}(z)
\]
with $R_{0}(z)=1$ and $R_{-n}(z)=0$, $\forall n\in\mathbb{N}$.
Equivalently the sequence $\left\{ R_{n}(z)\right\} _{n=0}^{\infty}$
is generated by 
\[
\sum_{n=0}^{\infty}R_{n}(z)t^{n}=\frac{1}{1-(a_{1}z+b_{1})t-(a_{2}z+b_{2})t^{2}}.
\]
In this section, we study neccessary and sufficient conditions on
$a_{1}$, $b_{1}$, $a_{2}$, and $b_{2}$ under which all the zeros
of the polynomial 
\[
\sum_{k=0}^{n}R_{n-k}(z)
\]
are real. Those polynomials form a sequence whose generating function
is 
\begin{align*}
\sum_{n=0}^{\infty}\sum_{k=0}^{n}R_{k}(z)t^{n} & =\sum_{k=0}^{\infty}t^{k}\sum_{n=k}^{\infty}R_{n-k}(z)t^{n-k}\\
 & =\frac{1}{(1-t)\left(1-(a_{1}z+b_{1})t-(a_{2}z+b_{2})t^{2}\right)}.
\end{align*}
With the substitutions $t$ by $-t$, $a_{2}$ by $-a_{2}$, and $b_{2}$
by $-b_{2}$, and then substitute $a_{2}z+b_{2}$ by $z$, we reduce
the generating function to the form 
\[
\frac{1}{(t+1)((az+b)t^{2}+zt+1)}.
\]
Note that all the substitutions above preserve the reality of the
zeros of the generated sequence of polynomials. We state the main
theorem of this section. 
\begin{thm}
\label{thm:firsttheorem} Let $a,b\in\mathbb{R}$. The zeros of all
the polynomials $P_{n}(z)$ generated by 
\begin{equation}
\sum_{n=0}^{\infty}P_{n}(z)t^{n}=\frac{1}{(t+1)((az+b)t^{2}+zt+1)}\label{eq:genfuncPn}
\end{equation}
are real if and only if $b\geq1+2\left|a\right|.$ Under this condition
the zeros of $P_{n}(z)$ lie on 
\begin{equation}
\left(2a-2\sqrt{a^{2}+b},2a+2\sqrt{a^{2}+b}\right)\label{eq:intPn}
\end{equation}
and are dense there as $n\rightarrow\infty$. 
\end{thm}

\subsection{The sufficient condition}

We assume $b\ge1+2|a|$. To prove the zeros of $P_{n}(z)$ lie on
\eqref{eq:intPn}, we count the number of real zeros of $P_{n}(z)$
on this interval and show that this number is at least the degree
of this polynomial which is given by the lemma below. 
\begin{lem}
\label{lem:degPn}For each $n\in\mathbb{N}$, the degree of $P_{n}(z)$
is at most $n$. 
\end{lem}

\begin{proof}
We collect the coefficients in $t$ of the denominator of the right
side of \eqref{eq:genfuncPn} and obtain the recurrence 
\begin{equation}
P_{n}(z)=-(z+1)P_{n-1}(z)-((a+1)z+b)P_{n-2}(z)-(az+b)P_{n-3}(z)\label{eq:recurrencePn}
\end{equation}
where $P_{0}(z)=1$ and $P_{-n}(z)=0$, $\forall n\in\mathbb{N}$.
The lemma follows from induction. 
\end{proof}
To count the number of real zeros of $P(z)$, we construct two auxiliary
real-valued functions $z(\theta)$ and $\tau(\theta)$ on $\theta\in(0,\pi)$.
The first function is defined as 
\begin{equation}
z(\theta)=2a\cos^{2}\theta-2\cos\theta\sqrt{a^{2}\cos^{2}\theta+b}.\label{eq:zthetadef}
\end{equation}
By the quadratic formula, $z(\theta)$ satisfies 
\begin{equation}
z(\theta)^{2}-4az(\theta)\cos^{2}\theta-4b\cos^{2}\theta=0.\label{eq:zthetaeq}
\end{equation}

We will show later that there are $n$ values of $\theta\in(0,\pi)$,
each of which yields a zero of $P_{n}(z)$ on \eqref{eq:intPn} via
$z(\theta)$. The lemma below ensures a bijective correspondence between
$\theta$ and $z(\theta)$. 
\begin{lem}
\label{lem:zmonotone}The function $z(\theta)$ is increasing on $(0,\pi)$
and it maps this interval onto the interval 
\[
\left(2a-2\sqrt{a^{2}+b},2a+2\sqrt{a^{2}+b}\right).
\]
\end{lem}

\begin{proof}
To show $z(\theta)$ is increasing, we compute its derivative 
\[
\frac{dz}{d\theta}=-4a\cos\theta\sin\theta+\frac{4a^{2}\cos^{2}\theta\sin\theta+2b\sin\theta}{\sqrt{a^{2}\cos^{2}\theta+b}}
\]
and see that it suffices to show 
\[
2a^{2}\cos^{2}\theta+b>2\left|a\cos\theta\right|\sqrt{a^{2}\cos^{2}\theta+b}.
\]
The left side is positive and the squares of both sides reduce the
inequality to $b^{2}>0$, which shows that $z(\theta)$ is increasing.
We complete the lemma by computing the limits 
\begin{align*}
\lim_{\theta\rightarrow0}z(\theta) & =2a-2\sqrt{a^{2}+b},\\
\lim_{\theta\rightarrow\pi}z(\theta) & =2a+2\sqrt{a^{2}+b}.
\end{align*}
\end{proof}
To define the second function $\tau(\theta)$, we need the following
lemma. 
\begin{lem}
\label{lem:tauwelldef}For any $\theta\in(0,\pi)$, we have 
\[
az(\theta)+b>0.
\]
\end{lem}

\begin{proof}
From Lemma \ref{lem:zmonotone}, it suffices to show that 
\[
b+2a^{2}>2|a|\sqrt{a^{2}+b}.
\]
Since we know the left side is positive by $b\ge1+2|a|$, we obtain
the inequality above by squaring both sides. 
\end{proof}
From Lemma \ref{lem:tauwelldef}, we define the functions 
\begin{align*}
\tau(\theta) & =\frac{1}{\sqrt{az(\theta)+b}},\\
t_{1}(\theta) & =\tau(\theta)e^{-i\theta},\\
t_{2}(\theta) & =\tau(\theta)e^{i\theta},
\end{align*}
on $\theta\in(0,\pi)$. 
\begin{lem}
\label{lem:zerosdenom}For any $\theta\in(0,\pi)$, the two zeros
of 
\begin{equation}
(az(\theta)+b)t^{2}+z(\theta)t+1\label{eq:quadfactor}
\end{equation}
are $t_{1}(\theta)$ and $t_{2}(\theta)$. 
\end{lem}

\begin{proof}
We verify that $\tau(\theta)e^{\pm i\theta}$ satisfy the Vieta's
formulas. Indeed, we have 
\begin{equation}
t_{1}(\theta)t_{2}(\theta)=\tau(\theta)^{2}=\frac{1}{az(\theta)+b}\label{eq:prodelem}
\end{equation}
and 
\begin{align*}
t_{1}(\theta)+t_{2}(\theta) & =2\tau(\theta)\cos\theta=\frac{2\cos\theta}{\sqrt{az(\theta)+b}}.
\end{align*}
From \eqref{eq:zthetadef}, we note that 
\begin{equation}
z(\theta)\cos\theta<0\label{eq:signztheta}
\end{equation}
since $b>0$. As a consequence, we obtain 
\begin{equation}
\frac{2\cos\theta}{\sqrt{az(\theta)+b}}=\frac{-z(\theta)}{az(\theta)+b}\label{eq:sumelemztheta}
\end{equation}
by squaring both sides and applying \eqref{eq:zthetaeq}. 
\end{proof}
The lemma below shows that for each $\theta\in(0,\pi)$, the two zeros
of \eqref{eq:quadfactor} lie inside the unit ball. 
\begin{lem}
\label{lem:taule1}For any $\theta\in(0,\pi)$, we have $|\tau(\theta)|<1$. 
\end{lem}

\begin{proof}
From \eqref{eq:prodelem}, \eqref{eq:sumelemztheta}, and \eqref{eq:zthetadef},
it suffices to show 
\[
\sqrt{a^{2}\cos^{2}\theta+b}>1+a\cos\theta.
\]
If the right side is negative, the inequality is trivial. If not,
we square both sides and the inequality follows from 
\[
b\ge1+2|a|>1+2a\cos\theta.
\]
\end{proof}
For each $\theta\in(0,\pi)$, the Cauchy differentiation formula gives
\begin{align*}
P_{n}(z(\theta)) & =\frac{1}{2\pi i}\ointctrclockwise_{|t|=\epsilon}\frac{1}{(t+1)((az(\theta)+b)t^{2}+z(\theta)t+1)t^{n+1}}dt\\
 & =\frac{1}{2\pi i}\ointctrclockwise_{|t|=\epsilon}\frac{1}{(az(\theta)+b)(t+1)(t-t_{1}(\theta))(t-t_{2}(\theta))t^{n+1}}dt.
\end{align*}
We recall that $az(\theta)+b\ne0$ by Lemma \ref{lem:tauwelldef}.
If we integrate the integrand over the circle $Re^{it}$, $0\le t\le2\pi$,
and let $R\rightarrow\infty$, then the integral approaches $0$.
Thus the sum of $P_{n}(z(\theta))$ and the residues of the integrand
at the three simple poles $-1$, $t_{1}(\theta)$ and $t_{2}(\theta)$
is $0$. We compute these residue and deduce that $-(az(\theta)+b)P_{n}(z(\theta))$
equals to 
\[
\frac{(-1)^{n+1}}{(1+t_{1}(\theta))(1+t_{2}(\theta))}+\frac{1}{t_{1}(\theta)^{n+1}(1+t_{1}(\theta))(t_{1}(\theta)-t_{2}(\theta))}+\frac{1}{t_{2}^{n+1}(\theta)(1+t_{2}(\theta))(t_{2}(\theta)-t_{1}(\theta))}.
\]
We multiply this expression by $(1+t_{1}(\theta))(1+t_{2}(\theta))\tau(\theta)^{n+1}$,
which is nonzero $\forall\theta\in(0,\pi)$, and conclude $\theta$
is a zero of $P_{n}(z(\theta))$ if and only if it is a zero of 
\[
(-1)^{n+1}\tau(\theta)^{n+1}+\frac{1+\tau(\theta)e^{i\theta}}{(\tau(\theta)e^{-i\theta}-\tau(\theta)e^{i\theta})e^{-i(n+1)\theta}}+\frac{1+\tau(\theta)e^{-i\theta}}{(\tau(\theta)e^{i\theta}-\tau(\theta)e^{-i\theta})e^{i(n+1)\theta}}
\]
or equivalently a zero of 
\[
(-1)^{n+1}\tau(\theta)^{n+1}-\frac{\sin((n+1)\theta)/\tau(\theta)+\sin((n+2)\theta)}{\sin\theta}.
\]
With the trigonometric identity $\sin(n+2)\theta=\sin((n+1)\theta)\cos\theta+\cos((n+1)\theta)\sin\theta$,
we write the expression above as 
\begin{equation}
(-1)^{n+1}\tau(\theta)^{n+1}-\cos((n+1)\theta)-\frac{\sin((n+1)\theta)(\cos\theta+1/\tau(\theta))}{\sin\theta}.\label{eq:functheta}
\end{equation}

We note that if 
\[
\theta=\frac{k\pi}{n+1},\qquad1\le k\le n,
\]
then the sign of \eqref{eq:functheta} is $(-1)^{k+1}$ since $\tau(\theta)<1$
by Lemma \ref{lem:taule1}. By the intermediate value theorem, \eqref{eq:functheta}
has at least $n-1$ solution on $(\pi/(n+1),n\pi/(n+1))$. We also
note that as $\theta\rightarrow0$, the sign of \eqref{eq:functheta}
is negative since $\sin((n+1)\theta)/\sin\theta$ approaches $n+1$
and $\tau(\theta)<1$. Thus \eqref{eq:functheta} has another zero
on $(0,\pi/(n+1))$. From Lemma \ref{lem:zmonotone}, each zero in
$\theta$ of \eqref{eq:functheta} gives exactly one zero in $z$
of $P_{n}(z)$ on 
\[
\left(2a-2\sqrt{a^{2}+b},2a+2\sqrt{a^{2}+b}\right).
\]
Thus all the zeros of $P_{n}(z)$ lie on the interval above by the
fundamental theorem of algebra and Lemma \ref{lem:degPn}. The density
of the zeros of $P_{n}(z)$ as $n\rightarrow\infty$ on this interval
follows directly from the density of the solutions of \eqref{eq:functheta}
and the continuity of $z(\theta)$.

\subsection{The necessary condition }

In this section, we will show that if either 
\begin{enumerate}
\item $b\leq-1$ or 
\item $-1<b<1+2|a|,$ 
\end{enumerate}
then not all polynomials $P_{n}(z)$ are hyperbolic. By \cite[Theorem 1.5]{sokal},
it suffices to find $z^{*}\in\mathbb{C\backslash\mathbb{R}}$ such
that the zeros of 
\begin{equation}
(t+1)((az^{*}+b)t^{2}+z^{*}t+1)\label{eq:denomz*}
\end{equation}
are distinct and the two smallest (in modulus) zeros of this polynomial
have the same modulus. Note that every small neighborhood of such
$z^{*}$ will contain a zero of $P_{n}(z)$ for all large $n$ and
consequently $P_{n}(z)$ is not hyperbolic for all large $n$. For
more details on this application of the theorem, see \cite{tz}.

For the first case $b\le-1$, we let $\theta^{*}$ be any angle with
$a^{2}\cos^{2}\theta^{*}<-b$ and let $\tau^{*}$ be any zero of 
\[
b\tau^{2}-2a\tau\cos\theta^{*}-1.
\]
Note that $\tau^{*}\notin\mathbb{R}$ since 
\[
a^{2}\cos^{2}\theta^{*}+b<0
\]
and consequently $\tau^{*2}\notin\mathbb{R}$ by the definition of
$\tau^{*}$. With the note that $2a\tau^{*}\cos\theta^{*}+1$ is nonreal
(and thus nonzero), we choose 
\[
z^{*}=\frac{-2b\tau^{*}\cos\theta^{*}}{2a\tau^{*}\cos\theta^{*}+1}
\]
which is nonreal since $1/z^{*}\notin\mathbb{R}$. From the definitions
of $\tau^{*}$, $\theta^{*}$, and $z^{*}$ above, the two solutions
of 
\[
(az^{*}+b)t^{2}+z^{*}t+1=0
\]
are $\tau^{*}e^{\pm i\theta^{*}}$ since they satisfy the Vieta's
formulas 
\[
\tau^{*2}=\frac{1}{az^{*}+b}
\]
and 
\[
2\tau^{*}\cos\theta^{*}=-\frac{z^{*}}{az^{*}+b}.
\]
Since $\tau^{*}$ and $\overline{\tau^{*}}$ are solutions of $b\tau^{2}-2a\tau\cos\theta^{*}-1$,
we have $\tau^{*}\overline{\tau^{*}}=|\tau^{*}|^{2}=-1/b\le1$. Thus
the two smallest (in modulus) zeros of \eqref{eq:denomz*} equal in
modulus and we complete the case $b\le-1$.

We now consider the case $-1<b<1+2|a|$. We will find $z^{*}\notin\mathbb{R}$
so that the smaller (in modulus) zero of $(az^{*}+b)t^{2}+z^{*}t+1$
lie on the unit circle. The inequality $\left|2|a|-b|\right|<1$ implies
that 
\[
1+2|a|>|b|>|b|.|2|a|-b|
\]
and consequently 
\[
1-b^{2}+2|a|+2b|a|>0.
\]
We conclude there is $\theta^{*}\in(0,\pi)$ sufficiently close to
$0$ when $a\ge0$ or close to $\pi$ when $a<0$ such that 
\[
b^{2}-2ab\cos\theta^{*}<1+2a\cos\theta^{*}.
\]
With this choice of $\theta^{*}$, we have 
\begin{equation}
\frac{|be^{i\theta^{*}}-a|}{|ae^{i\theta^{*}}+1|}=\frac{b^{2}+a^{2}-2ab\cos\theta}{a^{2}+1+2a\cos\theta}<1.\label{eq:ineqtheta*}
\end{equation}
We define 
\[
z^{*}=\frac{-1-be^{2i\theta^{*}}}{ae^{2i\theta^{*}}+e^{i\theta^{*}}}
\]
and write 
\begin{equation}
(az^{*}+b)t^{2}+z^{*}t+1\label{eq:quaddenomz*}
\end{equation}
as $z^{*}(at^{2}+t)+bt^{2}+1$ to conclude that $e^{i\theta^{*}}$
is a zero of this polynomial. Since the product of the two zeros of
this polynomial is $1/(az^{*}+b)$, we claim that the other zero of
this polynomial is more than $1$ in modulus by showing that 
\[
\frac{1}{|az^{*}+b|}>1.
\]
Indeed, from the definition of $z^{*}$, this inequality is equivalent
to \eqref{eq:ineqtheta*}. We note that $z^{*}\notin\mathbb{R}$ since
a solution of \eqref{eq:quaddenomz*} is $e^{i\theta^{*}}\notin\mathbb{R}$
and the other solution is more than $1$ in modulus.

\section{Linear combination of chebyshev polynomials}

The goal this section is to study necessary and sufficient conditions
under which the zeros of \eqref{eq:linearcombCheb} are real. The
sequence $\left\{ Q_{n}(z)\right\} $ in \eqref{eq:linearcombCheb}
is generated by 
\begin{align*}
\sum_{n=0}^{\infty}Q_{n}(z)t^{n} & =\sum_{n=0}^{\infty}\sum_{k=0}^{n}(az+b)^{k}U_{n-k}(z)t^{n}\\
 & =\sum_{k=0}^{\infty}(az+b)^{k}t^{k}\sum_{n=k}^{\infty}U_{n-k}(z)t^{n-k}\\
 & =\frac{1}{(1+(az+b)t)(1-2zt+t^{2})}.
\end{align*}
With the substitution $z$ by $-z/2$ and then $-a/2$ by $a$, it
suffice to study the hyperbolicity of the sequence generated of polynomials
by 
\[
\frac{1}{(1+(az+b)t)(1+zt+t^{2})}.
\]
As a small digression of the main goal, we will prove following theorem
which states that the positivity of the $t^{2}$- coefficient in the
factor $1+zt+t^{2}$ is important to ensure the hyperbolicity of the
generated sequence of polynomials. 
\begin{thm}
\label{thm:secondtheorem} Suppose $a,b,c\in\mathbb{R}$ where $c\ne0$.
If $c\le0$, then not all the polynomials $P_{n}(z)$ generated by
\[
\frac{1}{((az+b)t+1)(ct^{2}+zt+1)}.
\]
are hyperbolic. 
\end{thm}

We note that if $c=0$, the sequence of generated polynomials satisfy
a three-term recurrence and their zeros have been studied in \cite{tran}.
Under the condition $c>0$, with the substitution $t\rightarrow t/\sqrt{c}$,
we can assume $c=1$. The following theorem settles the necessary
and sufficient conditions for the hyperbolicity of \eqref{eq:linearcombCheb}. 
\begin{thm}
\label{thm:thirdtheorem} Suppose $a,b\in\mathbb{R}$. The zeros of
all the polynomials $P_{n}(z)$ generated by 
\begin{equation}
\sum_{n=0}^{\infty}P_{n}(z)t^{n}=\frac{1}{((az+b)t+1)(t^{2}+zt+1)}.\label{eq:genfuncthirdthm}
\end{equation}
are real if and only if $\left|b\right|\leq1-2\left|a\right|$. Moreover
when $|b|\le1-2|a|$, the zeros of $P_{n}(z)$ lies on $(-2,2)$ and
are dense there as $n\rightarrow\infty$. 
\end{thm}

\subsection{Proof of Theorem \ref{thm:secondtheorem}}

In the case $c<0$, with the substitution $t\rightarrow t/\sqrt{|c|}$,
it suffices to show that for any $a,b\in\mathbb{R}$, not all the
polynomials generated by 
\[
\frac{1}{((az+b)t+1)(-t^{2}+zt+1)}
\]
are hyperbolic. Recall a consequence of \cite[Theorem 1.5]{sokal}
that we will need to find $z^{*}\notin\mathbb{R}$ so that the two
smallest zeros of 
\[
((az^{*}+b)t+1)(-t^{2}+z^{*}t+1)
\]
equal in modulus.

In the case $|b|<1$, we choose $z^{*}=iy^{*}$ where 
\[
0<y^{*}<\min\left(\frac{\sqrt{1-b^{2}}}{|a|},2\right)
\]
if $a\ne0$ and $0<y^{*}<2$ if $a=0$. The two zeros of $-t^{2}+z^{*}t+1$,
\[
\frac{iy^{*}\pm\sqrt{4-y^{*2}}}{2}
\]
lie on the unit circle and thus their modulus is less than 
\[
\frac{1}{|az^{*}+b|}=\frac{1}{\sqrt{a^{2}y^{*2}+b^{2}}}.
\]

For the remainder of Section 3.1, we assume $|b|\ge1$. To make a
suitable choice for $z^{*}$, we consider the following lemma. 
\begin{lem}
\label{lem:theta*existence}With the principal cut, there exists $\theta^{*}\ne k\pi$,
$k\in\mathbb{Z}$, such that 
\[
|b|+\sqrt{b^{2}+4a^{2}-4ae^{i\theta^{*}}}\geq\left|2a-2e^{i\theta^{*}}\right|.
\]
\end{lem}

\begin{proof}
We note that $b^{2}+4a^{2}\ge4|a|$ since 
\[
4|a|(1-|a|)\le1\le|b|.
\]
Thus with the principle cut, the function 
\[
f(z):=\frac{|b|+\sqrt{b^{2}+4a^{2}-4az}}{2a-2z}
\]
is meromorphic on the open unit ball with the possible pole at $z=a$
if $|a|<1.$ To prove this lemma, we will find $z\notin\mathbb{R}$
and $|z|=1$ such that $|f(z)|\ge1$.

We note that if $|a|\ge1$, then $f(z)$ is analytic on the unit ball
and 
\[
|f(0)|=\frac{|b|+\sqrt{b^{2}+4a^{2}}}{2|a|}>1.
\]
Thus by the maximum modulus principle $|f(z)|>1$ for some $|z|=1$.
We can choose such $z\notin\mathbb{R}$ by the continuity of $f(z)$.

On the other hand if $|a|<1$, then the Cauchy integral formula implies
that 
\[
\ointctrclockwise_{|z|=1}|f(z)||dz|\ge\left|\ointctrclockwise_{|z|=1}f(z)dz\right|=2\pi|b|\ge2\pi.
\]
Consequently $|f(z)|>1$ for some $|z|=1$ or $|f(z)|=1$ for all
$|z|=1$ and the lemma follows. 
\end{proof}
We now define 
\[
z^{*}=\frac{-2ab+be^{i\theta^{*}}+\sign(b)e^{i\theta^{*}}\sqrt{b^{2}+4a^{2}-4ae^{i\theta^{*}}}}{2a^{2}-2ae^{i\theta^{*}}}
\]
where $\theta^{*}$ is given in Lemma \ref{lem:theta*existence}.
With this definition, $z^{*}$ is a solution of 
\[
(a^{2}-ae^{i\theta^{*}})z^{2}+(2ab-be^{i\theta^{*}})z+b^{2}-e^{2i\theta^{*}}=0
\]
from which we deduce that 
\begin{equation}
-\frac{e^{i\theta^{*}}}{az^{*}+b}=-\frac{2a-2e^{i\theta^{*}}}{-b+\sign(b)\sqrt{b^{2}+4a^{2}-4ae^{i\theta^{*}}}}\label{eq:firstzero}
\end{equation}
is a zero in $t$ of 
\[
-t^{2}+z^{*}t+1.
\]
The modulus of \eqref{eq:firstzero} is the same as the modulus of
the zero in $t$ of $(az^{*}+b)t+1$ which is at most $1$ by the
definition of $\theta^{*}$. This modulus is larger than the modulus
of the other zero of $-t^{2}+z^{*}t+1$ since the product of two zeros
of this polynomial is $-1$. We finish the proof of Theorem \ref{thm:secondtheorem}
by noting that $z^{*}\notin\mathbb{R}$ since the two zeros of $-t^{2}+z^{*}t+1$
are neither real nor complex conjugate.

\subsection{Proof of Theorem \ref{thm:thirdtheorem}}

\subsubsection{The sufficient condition}

Let $\left\{ P_{n}(z)\right\} $ be the sequence of polynomials defined
in \eqref{eq:genfuncthirdthm} where $\left|b\right|\leq1-2\left|a\right|$.
The proof of the following lemma is the same as that of Lemma 4 in
\cite{tz}. For brevity, we omit the proof in this paper. 
\begin{lem}
\label{lem:adensity}For each $b\in[-1,1],$ let $S_{b}$ be a dense
subset of 
\begin{equation}
\left[\frac{|b|-1}{2},\frac{1-|b|}{2}\right]\label{eq:ainterval}
\end{equation}
and $n\in\mathbb{N}$ be fixed. If for any $a\in S_{b}$, the zeros
of $P_{n}(z)$ lie on $(-2,2)$, then the same conclusion holds for
any $a$ in \eqref{eq:ainterval} . 
\end{lem}

Suppose $\left|b\right|\leq1-2\left|a\right|$. From Lemma \ref{lem:adensity},
it suffices to consider $a\ne0$. We define the monotone function
$z(\theta)=-2\cos\theta$ on $(0,\pi)$ and note that for each $\theta\in(0,\pi)$
the two zeros of $t^{2}+z(\theta)t+1$ are $e^{\pm i\theta}$. We
consider the function 
\[
t_{0}(\theta)=\frac{-1}{az(\theta)+b},\qquad\theta\in(0,\pi),
\]
which has a vertical asymptote at $\theta=\cos^{-1}(b/2a)$ if $|b|<2|a|$.
For any $\theta\in(0,\pi)$ such that $2a\cos\theta\ne b$, the Cauchy
differentiation formula gives 
\[
P_{n}(z(\theta))=\frac{1}{az(\theta)+b}\ointctrclockwise_{|t|=\epsilon}\frac{dt}{(t-t_{0}(\theta))(t-e^{i\theta})(t-e^{-i\theta})t^{n+1}}.
\]
After computing the residue of the integrand at the three nonzero
simple poles $t_{0}(\theta),e^{\pm i\theta}$, and letting the radius
of the integral approach infinity, we apply similar computations in
\eqref{eq:functheta} to conclude that $\theta\in(0,\pi)$, $2a\cos\theta\ne b$,
is a zero of $P_{n}(z(\theta))$ if and only if it is a zero of 
\begin{equation}
\frac{-1}{t_{0}(\theta)^{n+1}}+\cos\left((n+1)\theta\right)+\frac{(\cos\theta-t_{0}(\theta))\sin\left((n+1)\theta\right)}{\sin\theta}.\label{eq:secondfunctheta}
\end{equation}
From Lemma \ref{lem:adensity}, it suffices to consider $|b|\ne2|a|$.
We note that the limits of \eqref{eq:secondfunctheta} as $\theta\rightarrow0$
and $\theta\rightarrow\pi$ are 
\begin{equation}
n+2+\frac{n+1}{b-2a}+(-1)^{n}(b-2a)^{n+1}\label{eq:leftlimit}
\end{equation}
and 
\begin{equation}
(-1)^{n+1}(n+2)+(-1)^{n}\left(\frac{n+1}{b+2a}+(b+2a)^{n+1}\right)\label{eq:rightlimit}
\end{equation}
respectively.

In the case $|b|>2|a|$, \eqref{eq:secondfunctheta} is a continuous
function of $\theta$ on $(0,\pi)$ and its sign at $\theta=k\pi/(n+1)$,
for $1\le k\le n$, is $(-1)^{k}$ since 
\[
|t_{0}(\theta)|>\frac{1}{2|a|+|b|}\ge1.
\]
By the intermediate value theorem, we obtain at least $n-1$ zeros
of \eqref{eq:secondfunctheta} on $(\pi/(n+1),n\pi/(n+1))$. If $b>0$,
then \eqref{eq:leftlimit} is positive since $0<b-2a\le1$ and we
obtain at least another zero of \eqref{eq:secondfunctheta} on $(0,\pi/(n+1))$.
On the other hand, if $b<0$, then the inequalities 
\[
-1<b+2a<0
\]
imply that the sign of \eqref{eq:rightlimit} is $(-1)^{n+1}$ and
we have at least another zero of \eqref{eq:secondfunctheta} on $(n\pi/(n+1),\pi)$.
We conclude that when $|b|>2|a|$, \eqref{eq:secondfunctheta} has
at least $n$ zeros on $(0,\pi)$, each of which yields a zero of
$P_{n}(z)$ on the interval $(-2,2)$ by the map $z(\theta)$. Thus
all the zeros of $P_{n}(z)$ lie on $(-2,2)$ by the fundamental theorem
of algebra.

We now consider the case $|b|<2|a|$. As a function of $\theta$ on
$(0,\pi)$, \eqref{eq:secondfunctheta} has a vertical asymptote at
$\theta=\cos^{-1}(b/2a)$ since $t_{0}(\theta)$ does. By Lemma \ref{lem:adensity},
we can assume 
\[
\cos^{-1}\frac{b}{2a}\ne\frac{k\pi}{n+1},\qquad1\le k\le n.
\]
Thus for some $0\le k_{0}\le n$, the open interval 
\begin{equation}
\left(\frac{k_{0}}{n+1}\pi,\frac{k_{0}+1}{n+1}\pi\right)\label{eq:asympint}
\end{equation}
contains $\cos^{-1}\left(b/2a\right)$. We note that this interval
may or may not contain a zero of \eqref{eq:secondfunctheta}. In the
case $a<0$, we observe that \eqref{eq:leftlimit} is positive and
the sign of \eqref{eq:rightlimit} is $(-1)^{n+1}$. Thus there are
at least $n$ zeros of \eqref{eq:secondfunctheta} on the $n$ intervals
$(k\pi/(n+1),(k+1)\pi/(n+1))$, for $0\le k\le n$ and $k\ne k_{0}$
and we conclude all the zeros of $P_{n}(z)$ lie on $(-2,2)$ by the
same argument in the previous case. On the other hand, if $a>0$,
then the limits \eqref{eq:secondfunctheta} as $\theta$ approaches
the left and right of $\cos^{-1}(b/2a)$ are 
\[
\lim_{\theta\to\cos^{-1}(b/2a)^{-}}\frac{\sin((n+1)\theta)}{b-2a\cos(\theta)}=(-1)^{k_{0}+1}\infty
\]
and 
\[
\lim_{\theta\to\cos^{-1}(b/2a)^{+}}\frac{\sin((n+1)\theta)}{b-2a\cos(\theta)}=(-1)^{k_{0}}\infty,
\]
respectively. If $k_{0}\ne0$ and $k_{0}\ne n$, then we conclude
that \eqref{eq:asympint} contains at least two zeros of \eqref{eq:secondfunctheta}.
Thus we obtain at least $n$ zeros of this expression on the $n-1$
intervals $(k\pi/(n+1),(k+1)\pi/(n+1))$, for $1\le k<n$. In the
case $k_{0}=0$ or $k_{0}=n$, \eqref{eq:asympint} contains at least
one zero of \eqref{eq:secondfunctheta} and thus there are at least
$n$ zeros of \eqref{eq:secondfunctheta} on the $n$ intervals $(k\pi/(n+1),(k+1)\pi/(n+1))$,
for $1\le k<n$ and $k=k_{0}$.

\subsubsection{The necessary condition}

In this section, we assume $|b|+2|a|>1$ and show that not all zeros
of $P_{n}(z)$ defined in \eqref{eq:genfuncthirdthm} are real when
$n$ is large. From \cite[Theorem 1.5]{sokal} , it suffices find
$z\notin\mathbb{R}$ so that $|t_{0}|=|t_{1}|\le|t_{2}|$ where 
\begin{equation}
t_{0}:=-\frac{1}{az+b}\label{eq:t0choice}
\end{equation}
and $t_{1}$ and $t_{2}$ are the two zeros of $1+zt+t^{2}$. To motivate
the choice of $z$, we provide heuristic arguments by noticing that
$t_{1}t_{2}=1$ and letting 
\begin{equation}
t_{1}=t_{0}e^{i\theta}=-\frac{e^{i\theta}}{az+b}\label{eq:t1choice}
\end{equation}
\begin{equation}
t_{2}=-e^{-i\theta}(az+b).\label{eq:t2choice}
\end{equation}
The equation $1+zt_{2}+t_{2}^{2}=0$ yields 
\[
(az+b)^{2}-ze^{i\theta}(az+b)+e^{2i\theta}=0
\]
or equivalently 
\begin{multline}
(a^{2}-ae^{i\theta})z^{2}+(2ab-be^{i\theta})z+b^{2}+e^{2i\theta}=0.\label{eq:zquadratic}
\end{multline}
With a choice of branch cut which will be specified later, the equation
above has two solutions 
\[
z=\frac{-2ab+be^{i\theta}\pm e^{i\theta}\sqrt{b^{2}-4a^{2}+4ae^{i\theta}}}{2a^{2}-2ae^{i\theta}}
\]
and the corresponding values for $az+b$ are

\begin{equation}
az+b=\frac{-be^{i\theta}\pm e^{i\theta}\sqrt{b^{2}-4a^{2}+4ae^{i\theta}}}{2a-2e^{i\theta}}.\label{eq:az+b}
\end{equation}
For a formal proof of the necessary condition, we consider the following
cases.

\textbf{Case 1: $|a|\le1$.} We have the inequality 
\[
b^{2}-4a^{2}+4|a|-(|b|+2|a|-2)^{2}=4(1-|a|)(2|a|+|b|-1)\ge0.
\]
with equality if and only if $|a|=1$. This implies 
\begin{equation}
b^{2}-4a^{2}+4|a|\ge0\label{eq:firstineqabale1}
\end{equation}
and 
\begin{align}
\sqrt{b^{2}-4a^{2}+4|a|}+|b| & \ge\left||b|+2|a|-2\right|+|b|\nonumber \\
 & \ge|2|a|-2|\label{eq:secineqabale1}
\end{align}
with equality if and only if $|a|=1$ and $b=0$. We define $\theta\in(0,\pi)$
sufficiently close to $0$ or $\pi$ such that $e^{i\theta}$ is close
to $\sign a$ if $a\ne0$. If $a=0$, we pick any $\theta\in(0,\pi)$.
With this choice of $\theta$ and the principal cut, we let 
\begin{equation}
z=\begin{cases}
\frac{-2ab+be^{i\theta}-\sign b.e^{i\theta}\sqrt{b^{2}-4a^{2}+4ae^{i\theta}}}{2a^{2}-2ae^{i\theta}} & \text{ if }ab\ne0,\\
\frac{ie^{i\theta}}{\sqrt{a^{2}-ae^{i\theta}}} & \text{ if }b=0,\\
\frac{b^{2}+e^{2i\theta}}{be^{i\theta}} & \text{if }a=0.
\end{cases}\label{eq:zchoiceale1}
\end{equation}
With this choice of $z$, \eqref{eq:zquadratic} holds and consequently
$t_{1}$ and $t_{2}$ defined in \eqref{eq:t1choice} and \eqref{eq:t2choice}
are the zeros of $1+zt+t^{2}$. If $a=0$, then 
\begin{equation}
|t_{0}|=|t_{1}|<|t_{2}|\label{eq:t0t1t2}
\end{equation}
since $|b|>1$. If $b=0$ then the inequalities $|a|\le1$ and \eqref{eq:firstineqabale1}
imply that $|a|=1$. As a consequence, \eqref{eq:t0t1t2} follows
from \eqref{eq:t0choice}, \eqref{eq:t1choice}, \eqref{eq:t2choice},
and \eqref{eq:zchoiceale1}. Finally, if $ab\ne0$, then from \eqref{eq:az+b}
and \eqref{eq:secineqabale1}, we conclude $|az+b|$ approaches 
\[
\frac{|b|+\sqrt{b^{2}-4a^{2}+4|a|}}{2-2|a|}>1
\]
as $e^{i\theta}\rightarrow\sign(a)$. Thus from \eqref{eq:t1choice}
and \eqref{eq:t2choice} there is $\theta\in(0,\pi)$ sufficiently
close to $0$ or $\pi$ such that 
\[
|t_{0}|=|t_{1}|<|t_{2}|.
\]
We also note that $z\notin\mathbb{R}$ since if $z\in\mathbb{R}$,
then the fact that $t_{1},t_{2}\notin\mathbb{R}$ by \eqref{eq:t1choice}
and \eqref{eq:t2choice} implies $t_{1}=\overline{t_{2}}$ which contradicts
to $|t_{1}|<|t_{2}|$.

\textbf{Case 2: $|a|>1$ and $|b|<1$.} By the intermediate value
theorem there is $y\in(0,\infty)$ such that 
\[
2\sqrt{a^{2}y^{2}+b^{2}}-\sqrt{y^{2}+4}-y=0
\]
since the the left side is $2|b|-2<0$ when $y=0$ and its limit is
$\infty$ when $y\rightarrow\infty$. With the choice $z=iy$, we
have 
\[
|t_{0}|=\frac{1}{|az+b|}=\frac{1}{\sqrt{a^{2}y^{2}+b^{2}}}
\]
and the modulus of the smaller zero of $t^{2}+iyt+1$ is 
\[
\frac{\sqrt{y^{2}+4}-y}{2}=\frac{2}{\sqrt{y^{2}+4}+y}=|t_{0}|.
\]

\textbf{Case 3: $|b|\geq1$ and $|a|>1$.} If $2+|b|>2|a|$, then
with the same choice of $\theta$ and $z$ and the same argument as
in the first case, this case follows from 
\[
|\sqrt{b^{2}-4a^{2}+4ae^{i\theta}}+|b||>|b|>2|a|-2.
\]
We now consider\textbf{ $2+|b|\leq2|a|$.} We square both sides of
$2|a|-2\ge|b|$ to obtain 
\[
b^{2}-4a^{2}\le4-8|a|<-4|a|
\]
which implies that, with the cut $[0,\infty)$, the function 
\[
f(z):=\frac{-b+\sqrt{b^{2}-4a^{2}+4az}}{2a-2z}
\]
is analytic on a small region containing the closed unit ball. From
the maximum modulus principle and the fact that 
\begin{align*}
|f(0)| & =\frac{|-b+\sqrt{b^{2}-4a^{2}}|}{|2a|}\\
 & =1,
\end{align*}
we conclude there is $\theta\in\mathbb{R}$ so that $|f(e^{i\theta})|>1$.
With this $\theta$, we let 
\[
z=\frac{-2ab+be^{i\theta}+e^{i\theta}\sqrt{b^{2}-4a^{2}+4ae^{i\theta}}}{2a^{2}-2ae^{i\theta}}
\]
and apply \eqref{eq:t0choice}, \eqref{eq:t1choice}, \eqref{eq:t2choice},
and \eqref{eq:az+b} to conclude $|t_{0}|=|t_{1}|<|t_{2}|$. The fact
that $z\notin\mathbb{R}$ follows from the same argument in the previous
case.

\end{document}